\documentclass[notitlepage]{amsart}

\usepackage{amsfonts}
\usepackage[latin1]{inputenc}
\usepackage[all]{xypic}
\newtheorem{pro}{Proposition}[section]
\newtheorem{teo}[pro]{Theorem}
\newtheorem{defi}[pro]{Definition}
\newtheorem{lem}[pro]{Lemma}
\newtheorem{cor}[pro]{Corollary}
\newtheorem{remark}[pro]{Remark}

\newcommand{\pd}{{\mathrm{pd}}}
\newcommand{\mini}{{\mathrm{min}}}
\newcommand{\maxi}{{\mathrm{max}}}
\newcommand{\modu}{{\mathrm{mod}}}

\newcommand{\Ker}{{\mathrm{Ker}}}
\newcommand{\Coker}{{\mathrm{Coker}}}
\newcommand{\proj}{{\mathrm{proj}}}
\newcommand{\inj}{{\mathrm{inj}}}

\newcommand{\findim}{{\mathrm{fin.dim}}}
\newcommand{\gldim}{{\mathrm{gl.dim}}}
\newcommand{\Hom}{{\mathrm{Hom}}}
\newcommand{\rk}{{\mathrm{rk}}}
\newcommand{\End}{{\mathrm{End}}}
\newcommand{\Ext}{{\mathrm{Ext}}}
\newcommand{\Fidim}{{\phi\,\mathrm{dim}}}
\newcommand{\E}{{\mathcal{E}}}
\newcommand{\Z}{{\mathbb{Z}}}
\newcommand{\N}{{\mathbb{N}}}

\newcommand{\C}{{\mathcal{C}}}

\newcommand{\add}{{\mathrm{add}}}

\newcommand{\D}{{\mathrm{D}}}
\newcommand{\K}{{\mathrm{K}}}

\newcommand{\Com}{{\mathrm{C}}}

\newenvironment{dem}{\noindent {\bf Proof.}}{\hfill $\Box$\\}

\usepackage{latexsym,amssymb,amscd}
\usepackage{amsmath}
\usepackage[all]{xy}
\begin{document}
\title{ The $\phi$-dimension: A new homological measure}
\thanks{{\it{$2010$ Mathematics Subject Classifications.}} Primary 16E10. Secondary 16E35.\\
{\it{Keys words:}} finitistic dimension, Igusa-Todorov functions, derived categories}
\author{ S\^onia Fernandes,\\ Marcelo Lanzilotta,\\ Octavio Mendoza}
\date{}

\begin{abstract} In \cite{IT}, K. Igusa and G. Todorov introduced two functions $\phi$ and $\psi,$ 
which are natural and important homological measures generalising the notion of the projective 
dimension. These Igusa-Todorov functions have become into a powerful tool to understand better 
the finitistic dimension conjecture. 
\

In this paper, for an artin $R$-algebra $A$ and the Igusa-Todorov function $\phi,$ we characterise 
the $\phi$-dimension of $A$ in terms either of the bi-functors $\Ext^{i}_{A}(-, -)$ or Tor's bi-functors $\mathrm{Tor}^{A}_{i}(-,-).$ Furthermore, by using the first
characterisation of the $\phi$-dimension, we show that the finiteness of the $\phi$-dimension 
of an artin algebra is invariant under derived equivalences. As an application of this result, we 
generalise the classical Bongartz's result \cite[Corollary 1]{B} as follows: For an artin algebra 
$A,$ a tilting $A$-module $T$ and the endomorphism algebra $B=\End_A(T)^{op},$ we have that 
$\Fidim\,(A)-\pd\,T\leq \Fidim\,(B)\leq \Fidim\,(A)+\pd\,T.$
\end{abstract}

\maketitle
\section{Introduction}

In this paper we shall consider artin $R$-algebras and finitely generated left mo\-du\-les. For an artin algebra $A,$ we denote by $\modu\,(A)$ the category of finitely ge\-ne\-ra\-ted left $A$-modules. Furthermore, $\proj\,(A)$ denotes the class of finitely generated projective $A$-modules and
$\inj\,(A)$ denotes the class of finitely generated injective $A$-modules. Moreover, $\pd\,M$ stands for the projective dimension of any $M\in\modu\,(A).$ We recall that
 $$\findim\,(A):=\sup\{\pd\,M\;:\;M\in\modu\,(A)\text{ and } \pd\,M<\infty\}$$ is the so-called finitistic dimension of $A,$ and also that $$\gldim\,(A):=\sup\{\pd\,M\;:\;M\in\modu\,(A)\}$$ is the global dimension of $A.$ The interest in the finitistic dimension is because of the
 ``finitistic dimension conjecture", which is still open, and states that {\it{the finitistic dimension
 of any artin algebra is finite.}} This conjecture is closely related with several homological conjectures, and therefore it is a centrepiece for the development of the representation theory of artin algebras. The reader could see in \cite {HZ}, \cite{XiXu}, and references therein, for the development related with the finitistic dimension conjecture.
\

In \cite{IT}, K. Igusa and G. Todorov defined two functions, denoted by
$\phi$ and $\psi$, from objects in the category $\modu\,(A)$ to the natural numbers $\mathbb{N}.$ These 
Igusa-Todorov functions determine new homological measures, generalising the notion of projective dimension, 
and have become into a powerful tool to understand better the finitistic dimension conjecture. A lot of new 
ideas have been developed around the use of Igusa-Todorov functions \cite{HL,HLM,HLM2,HLM3,W,Wei,Xi1,Xi2,Xi3}. 
Recently, J. Wei  introduced in \cite{Wei} the 
notion of Igusa-Todorov algebras, which behaves nicely into the setting of Igusa-Todorov functions. In 
particular, Igusa-Todorov algebras have finite finitistic dimension. The class of Igusa-Todorov algebras 
contain many others, for example: algebras of representation dimension at most 3, algebras with radical cube 
zero, monomial algebras and left serial algebras. In fact, it is expected that all artin algebras are 
Igusa-Todorov.
\

The notion and the use of the Igusa-Todorov functions has been expanding to other settings, for example into 
the bounded derived category \cite{Xu} and finite dimensional co-modules for left semi-perfect co-algebras \cite{HaLM}. Our intention, in this paper, is to develop the theory of Igusa-Todorov 
functions for artin $R$-algebras. In particular, we try to convince the readers that these functions are natural and important 
homological measures. In this paper, we are dealing with the $\phi$-dimension.
\

Following \cite{HL} and \cite{HLM}, we recall that the $\phi$-dimension of an artin $R$-algebra $A$ is
$$\Fidim\,(A):=\sup\,\{\phi\,(M)\ :\ M\in\modu\,(A)\}.$$ Since the function $\phi$ is a refinement of the
projective dimension (see {\cite{IT}), we have   
$$\findim\,(A)  \leq \Fidim\,(A) \leq \gldim\,(A).$$ On the other hand,
 F. Huard and M. Lanzilotta proved in \cite{HL} that the Igusa-Todorov functions characterise self-injective algebras.
 They provided, in \cite{HL}, an example of an algebra $A$ showing that the global dimension of $A$ is not enough to
 determine whether $A$ is self-injective or not; and moreover, in such example they got that
 $\findim\,(A) < \Fidim\,(A) < \gldim\,(A).$ Therefore, by taking into account the previous
 discussion, we establish the following ``$\phi$-dimension conjecture": 
 \begin{center}{\it{The
 $\phi$-dimension of any artin algebra is finite.}}\end{center}
 Observe that, the 
 $\phi$-dimension conjecture implies the finitistic dimension conjecture; and hence it could be used as a tool to deal with the finitistic dimension conjecture.
\

In \cite{PX}, S. Pan and C. Xi proved that the finiteness of the finitistic dimension of left
coherent rings is preserved under derived equivalences. Inspired by \cite{PX} and \cite{K}, we
get in Section 4 the same result for the $\phi$-dimension of artin algebras (see Theorem \ref{tititi}). This result could be used in order to get a better understanding of both conjectures.
\

In Section 2, we introduce the necessary facts and notions needed for the de\-ve\-lo\-ping of the paper. In Section 3, we relate the  $\phi$-dimension, of an artin algebra $A,$ with the bi-functors $\Ext^{i}_{A}(-, -).$ In order to do that, we use the well known Auslander-Reiten formulas and Yoneda's Lemma. Here the main result is the Theorem \ref{FidimExt} and it has been the  main ingredient for the proof of Theorem \ref{tititi}.
\

In Section 4, we give the proof of the invariance (under derived equivalences) of the finiteness of the $\phi$-dimensions of artin algebras. That is, we prove (see Theorem \ref{tititi}) the following: If two artin algebras $A$ and $B$ are derived equivalent, then $\Fidim\,(A)< \infty$ if and only if $\Fidim\,(B) < \infty$.
More precisely, if $T^{\bullet}$ is a tilting complex over $A,$ with $n$ non-zero terms, such that $B\simeq\End_{D(\modu\,(A))}(T^{\bullet})^{op}$, then $\Fidim\,(A)-n \leq \Fidim\,(B) \leq \Fidim\,(A)+ n $. In particular (see Corollary \ref{coroPdT}), if $T^{\bullet}$ is given by a tilting module $T\in\modu\,(A)$ then $\Fidim\,(A)-\pd\,T \leq \Fidim\,(B) \leq \Fidim\,(A)+ \pd\,T $. Observe that this corollary is a generalisation of the classical Bongartz's result  \cite[Corollary 1]{B}. 

\section{Preliminaries}

Let $A$ be an artin $R$-algebra. We recall that
$\underline{\modu}\,(A)$ is the stable $R$-category modulo projectives,
whose objects are the same as in $\modu\,(A)$ and the morphisms are
given by
$\underline{\Hom}_A(M,\,N):=\Hom_A(M,\,N)/\mathcal{P}(M,\,N),$ where
$\mathcal{P}(M,\,N)$ is the $R$-submodule of $\Hom_A(M,\,N)$
consisting of the morphisms $M\to N$ factoring through objects in
$\proj\,(A).$ Similarly, we have the stable $R$-category modulo
injectives $\overline{\modu}\,(A)$ and the Auslander-Reiten translation
$\tau:\underline{\modu}\,(A)\to\overline{\modu}\,(A),$ which is an
$R$-equivalence of categories (see, for example, \cite{ARS}).
\

We start now by recalling the definition of the Igusa-Todorov
function $\phi:\mathrm{Obj}\,(\modu\,(A))\to\mathbb{N}.$ Let $\K(A)$ denote the quotient of
the free abelian group generated
by the set of iso-classes $\{[M]\;:\;M \in \modu(A)\}$
modulo the relations: (a) $[N]-[S]-[T]$ if $N \simeq S\oplus T$ and
(b) $[P]$ if $P$ is projective. Therefore, $\K(A)$ is the free
abelian group generated by the iso-classes of finitely generated
indecomposable non-projective $A$-modules. \

The syzygy functor $\Omega:\underline{\modu}\,(A)\to
\underline{\modu}\,(A)$ gives rise to a group homomorphism $\Omega:
\K(A) \to \K(A),$ where $\Omega([M]):=[\Omega(M)].$ Let $\langle M
\rangle$ denote the $\Z$-submodule of $\K(A)$ generated by the
indecomposable non-projective direct summands of $M.$ Since the rank
of $\Omega (\langle M \rangle)$ is less or equal than the rank of
$\langle M \rangle,$ which is finite, it follows from the well
ordering principle that there exists the smallest non-negative
integer $\phi (M)$ such  that $\Omega:\Omega^n(\langle M \rangle)\to
\Omega^{n+1}(\langle M \rangle)$ is an isomorphism for all $n\geq
\phi (M).$ Observe that $\phi(M)$ is always finite, whereas the projective dimension
$\pd\,M$ could be infinite.
\

The main properties of the Igusa-Todorov function $\phi$ are
summarised below.

\begin{lem}\label{itphi} \cite{IT, HLM} Let $A$ be an artin $R$-algebra and $M,N \in \modu(A).$ Then, the following
statements hold.
\begin{itemize}
\item[(a)] $\phi\,(M)=\pd\, M $ if $\pd\, M<\infty.$
\item[(b)] $\phi\,(M)=0$ if $M$ is indecomposable and $\pd\, M=\infty.$
\item[(c)] $\phi\,(M)\leq \phi\,(N\oplus M)$.
\item[(d)] $\phi\,(M)=\phi\,(N)$ if $\add\,(M)=\add\,(N).$
\item[(e)] $\phi\,(M\oplus P)=\phi\,(M)$ for any $P\in\proj\,(A).$
\item[(f)] $\phi(M) \leq \phi(\Omega\, M) +1.$
\end{itemize}
\end{lem}

It follows, from the above properties, that $\phi$ is a good refinement of the measure ``projective dimension". Indeed, for modules of finite projective dimension  both homological measures coincides; and in the case of infinite projective dimension, $\phi$ gives a finite number as a measure.

\section{$\phi$-dimension and the bi-functors $\Ext_A^{i}(-,-)$}

Let $A$ be an artin $R$-algebra. For a given $M\in\modu(A),$ the
projective cover of $M$ will be denoted by $\pi_M:P_0(M)\to M.$ We
also denote by $\C_A$ the abelian category of all $R$-functors
$F:\modu(A)\to\modu(R).$ For a given functor $F\in\C_A,$ the
isomorphism class of $F$ will be denoted by $[F],$ e.g.
$[F]:=\{G\in\C_A\;:\;G\simeq F\}.$ \

Finally we denote by $\underline{\mathcal{C}}_A$ the abelian category of all $R$-functors
$F:\underline{\modu}\,(A)\rightarrow\modu (R)$. Similarly, we introduce $\overline{\mathcal{C}}_A$
by using $\overline{\modu}\,(A)$ instead of $\underline{\modu}(A)$.
\

The following proposition  (exception by the item (d)) can be found in \cite[Proposition 1.44]{AB}. The proof 
given there, by M. Auslander and M. Bridger, uses one of the Hilton-Rees results (see \cite[Theorem 2.6]{HR}).  For the sake of completeness and the 
convenience of the reader, we include an elemental proof here, which is based on the Auslander-Reiten formula and Yoneda's Lemma.

\begin{pro} \label{eqExt} Let $A$ be an artin $R$-algebra and  $M,N\in\modu(A).$ Then, the following conditions are equivalent.
\begin{itemize}
 \item[(a)] $\Ext_A^{1}(M, -)\simeq\Ext_A^{1}(N, - )$ in $\C_A.$
 \item[(b)] $M\oplus P_0(N)\simeq N\oplus P_0(M)$ in $\modu(A).$
 \item[(c)] $M\simeq N$ in $\underline{\modu}\,(A).$
 \item[(d)] $[M]=[N]$ in $\K(A).$
\end{itemize}
\end{pro}
\begin{dem}(a) $\Leftrightarrow$ (c) We have that $\Ext^1_A(M,-)\simeq\Ext^1_A(N,-)$ in $\mathcal{C}_A$ if and only if, by using Auslander-Reiten
formula, the functors $D\overline{\Hom}_A(-, \tau M)$ and $D\overline{\Hom}_A(-, \tau N)$ are isomorphic in
$\overline{\mathcal{C}}(A)$. Moreover, the latest isomorphism is equivalent to the existence of an isomorphism $\underline{\Hom}_A(-, M)\simeq
\underline{\Hom}_A(-, N)$ in $\underline{\mathcal{C}}_A,$ since the Auslander-Reiten translation
$\tau:\underline{\modu}\,(A)\to\overline{\modu}\,(A)$ is an
$R$-equivalence of categories. Finally, the fact that $\underline{\Hom}_A(-, M)\simeq
\underline{\Hom}_A(-, N)$ in $\underline{\mathcal{C}}_A,$ is equivalent by Yoneda's Lemma to the existence of an isomorphism $M\simeq N$
in $\underline{\modu}\,(A)$.
\

(c) $\Rightarrow$ (b) Let $M\simeq N$ in $\underline{\modu}\,(A).$
Then we have the following diagram in $\modu\,(A)$
$$\xymatrix{  & & P_0(N) \ar[dr]^{\pi_N} &\\
M \ar[r]^\alpha \ar[dr]_{j_M} &  N \ar[ur]^{j_N} \ar[r]^\beta & M \ar[r]^\alpha & N\\
& P_0(M) \ar[ur]_{\pi_M} & &                        }$$ where
$1_M-\beta\alpha=\pi_M j_M$ and $1_N-\alpha\beta=\pi_N j_N.$ Furthermore, since $\pi_M$ and $\pi_N$ are epimorphisms, there exist
morphisms $f_N:P_0(N)\to P_0(M)$ and $f_M:P_0(M)\to P_0(N)$ in
$\modu\,(A)$ satisfying the equalities $\pi_M f_N=\beta\pi_N$ and
$\pi_N f_M=\alpha\pi_M.$ So we get the following diagram in
$\modu\,(A)$
$$M\oplus P_0(N)\stackrel{F}{\longrightarrow}
N\oplus P_0(M)\stackrel{G}{\longrightarrow}M\oplus P_0(N),$$ where
$F:=\begin{pmatrix} \alpha & \pi_N\\ j_M & -f_N
\end{pmatrix}$ and $G:=\begin{pmatrix} \beta & \pi_M\\ j_N & -f_M
\end{pmatrix}.$ We assert that $FG$ and $GF$ are isomorphisms.
Indeed, let $\mu:=j_M\pi_M+f_Nf_M \in \End_A(P_0(M))$ and so, from
the
above equalities,  we have that $FG=\begin{pmatrix}1_N & 0\\
j_M\beta-f_Nj_N & \mu\end{pmatrix}.$ Hence, in order to prove that
$FG$ is an isomorphism, it is enough to check that $\pi_M\mu=\pi_M$
since $\pi_M$ is a right-minimal morphism. That is,
$\pi_M\mu=\pi_Mj_M\pi_M+\pi_Mf_Nf_M=\pi_Mj_M\pi_M+\beta\pi_Nf_M=\pi_Mj_M\pi_M+\beta\alpha\pi_M=
(\pi_Mj_M+\beta\alpha)\pi_M=1_M\pi_M=\pi_M;$ proving that $FG$ is an
isomorphism. Analogously, it can be seen that $GF$ is an
isomorphism.
\

(b) $\Rightarrow$ (c) It is straightforward.
\

(c) $\Leftrightarrow$ (d) It follows from the fact that $\proj\,(A)$ is the iso-class of the zero object in the stable
 category $\underline{\modu}\,(A).$
\end{dem}
\

For an artin $R$-algebra $A$, we denote by $\E(A)$ the quotient of the free abelian group generated by the iso-classes
$[\Ext_A^{1}(M, -)]$ in $\C_A,$ for all $M\in\modu\,(A),$  modulo the relations
$$[\Ext_A^{1}(N, -)] - [\Ext_A^{1}(X, -)] - [\Ext_A^{1}(Y, -)] \quad\text{if}\quad N \simeq X\oplus Y.$$
The syzygy functor $\Omega:\underline{\modu}\,(A)\to \underline{\modu}\,(A)$ gives rise to a group homomorphism $\Omega:
\E(A) \to \E(A),$ given by $\Omega([\Ext_A^{1}(M, -)]):=[\Ext_A^{1}(\Omega\,M, -)].$

\begin{teo} \label{teogrupos} Let $A$ be an artin $R$-algebra. Then, the following statements hold.
 \begin{itemize}
 \item[(a)] The map
$\varepsilon:\K(A)\to \E(A),$ given by $\varepsilon([M]):=
[\Ext_{A}^{1}(M, -)],$ is an isomorphism of abelian groups and the
following diagram is commutative
$$\xymatrix{ \K(A) \ar[r]^\varepsilon \ar[d]_\Omega & \E(A) \ar[d]^\Omega \\
  \K(A) \ar[r]_\varepsilon & \E(A).}$$
 \item[(b)] $\varepsilon(\Omega^n([M]))=[\Ext_A^{n+1}(M,-)]$ for any $n\in\mathbb{N}$ and $M\in\modu\,(A).$
 \end{itemize}
\end{teo}
\begin{dem} (a) Since the bi-functor $\Ext_A^1(-,-)$ commutes
with finite direct sums and takes, in the first variable,
objects of $\proj\,(A)$ to zero, we get that the map $\varepsilon:\K(A)\to \E(A)$ is well defined, surjective
and it is also a morphism of abelian groups. Let
$M,N\in\modu\,(A),$ we assert that the equality $\varepsilon\,([M])=\varepsilon\,([N])$ implies that 
$[M]=[N]$ in $\K(A).$ Indeed, this assertion follows easily from Proposition \ref{eqExt}.
\

We prove now that $\varepsilon$ is a monomorphism. Let $x\in\Ker\,(\varepsilon).$ Since $\K(A)$ is the free 
abelian group with generators $[M],$ for $M\in\underline{\mathrm{ind}}\,(A)$, we have that 
$x=\sum_{i=1}^n\,a_i[M_i]$ and, without lost of generality, it can be assumed that the integer $a_i\neq 0$ for 
any $i.$ Using that $\varepsilon(x)=0,$ we get the equality 
$$\sum_{a_i>0}\,a_i\varepsilon([M_i])=\sum_{a_i<0}\,-a_i\varepsilon([M_i]).$$
Therefore $\displaystyle\varepsilon([\bigoplus_{a_i>0}\,M_i^{a_i}])=\varepsilon([\bigoplus_{-a_i>0}\,M_i^{-a_i}])$ and by the 
assertion above $$\displaystyle[\bigoplus_{a_i>0}\,M_i^{a_i}]=[\bigoplus_{-a_i>0}\,M_i^{-a_i}].$$ Hence 
$x=\sum_{i=1}^n\,a_i[M_i]=0,$ proving that $\varepsilon$ is an isomorphism.
Finally, we have $\varepsilon(\Omega\,([M])=[\Ext_A^1(\Omega\,M,-)]=\Omega(\varepsilon([M]).$
\

(b)
$\varepsilon(\Omega^n([M])=[\Ext_A^1(\Omega^n\,M,-)]=[\Ext_A^{n+1}(M,-)].$
\end{dem}

\begin{cor}\label{coroteogrupos} Let $A$ be an artin $R$-algebra and  $M,N\in\modu(A).$ Then, the following conditions are equivalent.
\begin{itemize}
 \item[(a)] $[\Omega^n M]=[\Omega^n N]$ in $\K(A).$
 \item[(b)] $\Ext_A^{i}(M, -)\simeq\Ext_A^{i}(N, - )$ in $\C_A$ for
 any $i\geq n+1.$
 \item[(c)] $\Ext_A^{n+1}(M, -)\simeq\Ext_A^{n+1}(N, - )$ in $\C_A.$
  \end{itemize}
\end{cor}
\begin{dem} It follows easily from  Theorem \ref{teogrupos}.
\end{dem}

\begin{defi} Let $A$ be an artin $R$-algebra, $d$ be a positive integer and $M$ in $\modu\,(A)$. A pair $(X,Y)$ of objects in $\add\,(M)$ is called a $d$-Division of $M$ if the following three conditions hold: 
\begin{itemize}
\item[(a)] $\add\,(X)\cap\add\,(Y)=\{0\};$
\item[(b)] $\Ext_A^{d}(X,-)\not\simeq \Ext_A^{d}(Y,-)$ in $\C_A;$
\item[(c)] $\Ext_A^{d+1}(X,-)\simeq \Ext_A^{d+1}(Y,-)$ in $\C_A.$
\end{itemize}
\end{defi}

\begin{remark} Observe that  $\phi\,(M)=0$ if and only if for any pair
$(X,Y)$ of objects in $\add\,(M),$ which are not projective and
$\add\,(X)\cap\add\,(Y)=\{0\},$ we have that $\Ext_A^{d}(X,
-)\not\simeq\Ext_A^{d}(Y, - )$ in $\C_A$ for
 any $d\geq 1.$ Thus, in this case, the following set is empty
 $$\{d\in \N\;:\;\text{there is a d-Division of }M \}.$$
\end{remark}

The following result gives a characterization of $\phi\,(M)$ in terms of the bi-functors $\Ext_A^{i}(-, -).$

\begin{teo} \label{FidimExt} Let $A$ be an artin $R$-algebra and $M$ in $\modu\,(A)$. Then
$$\phi\,(M)=\maxi\,(\{d\in\N\;:\;\text{there is a $d$-Division of } M \}\cup \{0\}).$$
\end{teo}
\begin{dem} Let $n:=\phi\,(M)>0$ and $M=\oplus_{i=1}^t\,M_i^{m_i}$ be a direct sum decomposition of $M,$ where
$M_i\not\simeq M_j$ for $i\neq j$ and $M_i$ is indecomposable for any
$i.$ Since $n$ is the first moment from which the rank of each free
abelian group of the family $\{\Omega^j(\langle M\rangle)\;:\;j\in\N\}$ becomes
stable, e.g. $$\phi\,(M)=\mini\,\{m\in\N\;:\;\rk\,\Omega^j(\langle M\rangle)=\rk\,\Omega^m(
\langle M\rangle)\;\forall\,j\geq m\},$$ we get the existence of natural numbers
$\alpha_1,\cdots,\alpha_t$ (not all zero) and a partition
$\{1,2,\cdots, t\}=I\uplus J$ such that $\sum_{i\in I}\alpha_i[\Omega^n M_i]=\sum_{j\in J}\alpha_j[\Omega^n M_j],$ and furthermore\\
$\sum_{i\in I}\alpha_i[\Omega^{n-1} M_i]\neq \sum_{j\in J}\alpha_j[\Omega^{n-1} M_j].$

Hence, by Corollary \ref{coroteogrupos}, it follows that the pair $(X,Y),$ with $X:=\oplus_{i\in I}M_i^{\alpha_i}$ and
$Y:=\oplus_{j\in J}M_j^{\alpha_j},$ is an $n$-Division of $M.$ Moreover, by using the fact that
$\rk\,\Omega^j(\langle M\rangle)>\rk\,\Omega^n(\langle M\rangle)$ for $j=0,1,\cdots, n-1,$ we get the result.
\end{dem}

For any  non-negative integer $i$ and any $M,N\in\modu\,(A),$ it is well known \cite[Proposition 5.3]{CE}
the existence of an isomorphism $$\mathrm{Tor}^{A}_{i}(D(M),N))\simeq D\,\Ext^{i}_A(N,M),$$ which is natural 
in both variables. Hence, a $d$-Division can be given in terms of Tor's functors as follows.

\begin{remark}\label{ul} Let $A$ be an artin $R$-algebra, $d$ be a positive integer and $M$ in $\modu\,(A)$. Then, a 
pair $(X,Y)$ of objects in $\add\,(M)$ is a $d$-Division of $M$ if the following three conditions hold: 
\begin{itemize}
\item[(a)] $\add\,(X)\cap\add\,(Y)=\{0\};$
\item[(b)] $\mathrm{Tor}^{A}_{d}(-,X)\not\simeq \mathrm{Tor}^{A}_{d}(-,Y)$ in $\C_{A^{op}};$
\item[(c)] $\mathrm{Tor}^{A}_{d+1}(-,X)\simeq \mathrm{Tor}^{A}_{d+1}(-,Y)$ in $\C_{A^{op}}.$
\end{itemize}
\end{remark}

As a consequence of Remark \ref{ul} and Theorem \ref{FidimExt}, we get that the Igusa-Todorov function $\Phi$ can also be characterised by using the Tor's bi-functors $\mathrm{Tor}^{A}_{i}(-,-).$

\section{Invariance of the $\phi$-dimension}

Let $\mathcal{A}$ be a full additive subcategory of an abelian category $\mathfrak{A}.$ A complex $X^{\bullet}$ over $\mathcal{A}$ is a sequence of morphisms
$\{d_{X}^{i}:X^{i}\to X^{i+1}\}_{i\in\Z}$ in $\mathcal{A},$ called the differentials of the complex $X^{\bullet},$ such that $d_{X}^{i}d_{X}^{i-1}=0$
for all $i \in \mathbb{Z}$. We write $X^{\bullet}= (X^{i}, d^{i}_{X})$ and
denote by $H^{i}(X^{\bullet})$ the $i$-th cohomology of the complex $X^{\bullet}$. Observe that $H^{\bullet}(X^{\bullet}),$ with zero
differentials, is also a complex over $\mathfrak{A}.$ Usually, a complex $X^{\bullet}$
is written as follows 
$$X^{\bullet}\;:\quad \cdots \longrightarrow X^{i-1} \stackrel{d_{X}^{i-1}}{\longrightarrow} X^{i} \stackrel{d_{X}^{i}}{\longrightarrow}X^{i+1} \longrightarrow \cdots$$

A complex $X^{\bullet}$ induces, in a natural way, the following complexes (called truncations)
$$\begin{matrix}
\tau_{\leq n} (X^{\bullet})\;:\quad \cdots \rightarrow X^{n-2} \rightarrow X^{n-1} \rightarrow X^{n} \rightarrow 0 \rightarrow \cdots ,\\{}\\
\tau_{\geq n} (X^{\bullet})\;:\quad \cdots \rightarrow 0 \rightarrow X^{n} \rightarrow X^{n+1} \rightarrow X^{n+2} \rightarrow \cdots .
\end{matrix}$$

 The category of all complexes over $\mathcal{A},$ with the usual complex maps of degree zero as morphisms, is denoted by $\Com\,(\mathcal{A})$. Hence, 
we have the $i$-th cohomology functor $H^i:\Com\,(\mathcal{A})\to\mathfrak{A}.$ For a given subset $\Xi$ of $\Z$, we consider the class of
complexes $$\Com\,(\mathcal{A})_{\Xi}:=\{\ X^{\bullet}\in \Com\,(\mathcal{A})\;:\;\ X^i=0\quad \forall\, i\not\in \Xi\}$$ which are concentrated on the set $\Xi.$ Usually, as $\Xi,$ we shall consider intervals of integer numbers of the form $[n,m]:=\{x\in\Z\;:\; n\leq x\leq m\}.$
\

We denote by $\K\,(\mathcal{A})$ the homotopy category of complexes over $\mathcal{A},$ and by $\K^{-}\,(\mathcal{A}),$ $\K^{+}\,(\mathcal{A})$ and
$\K^{b}\,(\mathcal{A})$ to the full triangulated subcategories of $\K\,(\mathcal{A})$ consisting, respectively, of the bounded above, bounded below and bounded complexes. In case $\mathcal{A}$ is
an abelian category, we denote by $\D\,(\mathcal{A})$ the derived category of complexes over 
$\mathcal{A},$ and by $\D^{-}\,(\mathcal{A}),$ $\D^{+}\,(\mathcal{A})$ and $\D^{b}\,(\mathcal{A})$ the 
full triangulated subcategories of $\D\,(\mathcal{A})$ consisting, respectively, of the complexes 
which are bounded above, bounded below and with bounded cohomology.
\

Let $\Lambda$ be an artin $R$-algebra. In this special case, we simplify the notation by writing $\D\,(\Lambda)$ instead of $\D\,(\modu\,(\Lambda)).$ Analogously, we write $\Com\,(\Lambda)$ instead of $\Com\,(\modu\,(\Lambda)).$ Furthermore, for any integers $a,b$ with 
$a\leq b,$ we denote by $\D\,(\Lambda)_{[a,b]}$ to the full subcategory of $\D\,(\Lambda)$ whose objects $X^\bullet$ are such that $X^i=0$ for $i\not\in[a,b].$ It is also well-known that the canonical functor $\imath_0:\modu\,(\Lambda)\to\D\,(\Lambda),$ which sends $M\in\modu\,(\Lambda)$ to the stalk complex $M[0]$ concentrated in degree zero, is additive full and faithful. Hence, through the functor $\imath_0,$  the module category  $\modu\,(\Lambda)$ can be considered as a full additive subcategory of $\D\,(\Lambda).$ Finally, for the sake of simplicity, we sometimes denote the bi-functor $\Hom_{\D\,(\Lambda)}( - , - )$ by $(- , - ).$
\

In order to prove the main result of this section, we start with the following preparatory lemmas. In all that follows, $\Lambda$ stands for an artin $R$-algebra.

\begin{lem} \cite[Lemma 1.6]{K}  \label{projetivo} Let $m,k,d \, \in \, \mathbb{Z}$ with $d \geq 0;$ and let $X^{\bullet}, Y^{\bullet} \in \K^{b}\,(\modu\,(\Lambda))$ be such that $X^{p}=0$ for all $p < m$ and $Y^{q}=0$ for all $q > k.$ If $\Ext^{i}_{\Lambda}(X^{r}, Y^{s})=0$ for all $r, s \in \mathbb{Z}$ and $i \geq d,$ then
$$\Hom_{\D\,(\Lambda)}(X^{\bullet}, Y^{\bullet}[i])= 0\quad\text{ for all }i \geq d+ k-m.$$
\end{lem}

\begin{lem} \label{lema3} Let $k\geq 0$ and $\ell>0;$ and
let $Z^{\bullet}, W^{\bullet}\in \Com\,(\Lambda)_{[-k,0]}$ with $Z^i, W^i\in \proj\,(\Lambda)\;$ $\forall\,i\in[-k+1, 0].$\\
If $\Hom_{\D\,(\Lambda)}(Z^{\bullet}, -[t])|_{\modu\,(\Lambda)}\simeq \Hom_{\D\,(\Lambda)}(W^{\bullet}, -[t])|_{\modu\,(\Lambda)}\;$ for all $t \geq k+\ell,$ then
$$\Ext_{\Lambda}^t(Z^{-k}, -) \simeq \Ext_{\Lambda}^t(W^{-k}, -)\;\text{ for all }\;t \geq 
k+\ell.$$
\end{lem}
\begin{proof} Assume that $\Hom_{\D\,(\Lambda)}(Z^{\bullet}, -[t])|_{\modu\,(\Lambda)}\simeq \Hom_{\D\,(\Lambda)}(W^{\bullet}, -[t])|_{\modu\,(\Lambda)}\;$ for all $t \geq k+\ell.$
Let  $S\in\modu\,(\Lambda).$ By applying the functor $(-, S[t])$ to the distinguished triangles
$$Z^{0}[0] \rightarrow Z^{\bullet} \rightarrow \tau_{\leq -1}(Z^{\bullet}) \rightarrow Z^{0}[1]\,\, \text{and} \,\,\,\, W^{0}[0]
\rightarrow W^{\bullet} \rightarrow \tau_{\leq -1}(W^{\bullet}) \rightarrow W^{0}[1],$$
\noindent we get the following exact sequences
$$(Z^{0}[0], S[t])  \rightarrow (\tau_{\leq -1}(Z^{\bullet})[-1], S[t]) \rightarrow (Z^{\bullet}[-1], S[t]) \rightarrow  (Z^{0}[-1], S[t]),$$
$$(W^{0}[0], S[t])  \rightarrow (\tau_{\leq -1}(W^{\bullet})[-1], S[t]) \rightarrow (W^{\bullet}[-1], S[t]) \rightarrow  (W^{0}[-1], S[t]).$$
\noindent Moreover, since $Z^{0},$ $W^{0}\in\proj\,(\Lambda)$ and $t\geq k+\ell\geq 1,$ it follows that\\ $(\tau_{\leq -1}(Z^{\bullet})[-1], -[t])|_{\modu\,(\Lambda)}
 \simeq (Z^{\bullet}[-1], -[t])|_{\modu\,(\Lambda)}$ and also we have that\\ $(\tau_{\leq -1}(W^{\bullet})[-1], -[t])
 |_{\modu\,(\Lambda)} \simeq (W^{\bullet}[-1], -[t])|_{\modu\,(\Lambda)}$ for any $t\geq n+\ell.$ Hence, by composing the given above isomorphisms of functors, we
get that
$$(\tau_{\leq -1}(Z^{\bullet})[-1], -[t])|_{\modu\,(\Lambda)}\simeq (\tau_{\leq -1}(W^{\bullet})[-1], -[t])|_{\modu\,(\Lambda)}\;\text{ for any }\; t\geq k+\ell.$$
\noindent We proceed by induction as follows. The step $i=1,$ is given by the above isomorphism of functors. Consider $1 < i \leq k-1$ and apply the functor $(-,S[t]),$ for each $S\in\modu\,(\Lambda),$ to the distinguished triangles
$Z^{-i}[0] \rightarrow  \tau_{\leq -i}(Z^{\bullet})[-i] \rightarrow \tau_{\leq -i-1}\,(Z^{\bullet})[-i] \rightarrow Z^{-i}[1]$ and
$W^{-i}[0] \rightarrow  \tau_{\leq -i}(W^{\bullet})[-i] \rightarrow \tau_{\leq -i-1}\,(W^{\bullet}) [-i] \rightarrow W^{-i}[1].$ So, by
repeating the same argument as above and using that\\ $(\tau_{\leq -i}(Z^{\bullet})[-i], -[t])
|_{\modu\,(\Lambda)}
\simeq (\tau_{\leq -i}(Z^{\bullet})[-i], -[t])|_{\modu\,(\Lambda)}$ (the previous step)
 for all $t \geq k+d$, the fact that $Z^{-i}$ and $W^{-i}$ are projective for all $-i\in[-k+1, 0]$, we obtain that 
$(Z^{-k}[0], -[t])|_{\modu\,(\Lambda)} \simeq ( W^{-k}[0], -[t])|_{\modu\,(\Lambda)}$ for all $t \geq k+d.$
\end{proof}

\begin{lem}\label{lema2} Let $k$, $\ell \in \mathbb{Z}$ be such that $0 < k< \ell.$ Let $Z^{\bullet}, W^{\bullet}\in \D\,(\Lambda)_{[-k,0]}$
with $Z^{i},$ $W^{i}\in\proj\,(\Lambda)$ for all $i\in[-k+1,0].$\\ If $\Hom_{\D\,(\Lambda)}(Z^{-k}[k], -\, [\ell])|_{\D\,(\Lambda)_{[-k,0]}} \simeq \Hom_{\D\,(\Lambda)}(W^{-k}[k], - \, [\ell])|_{\D\,(\Lambda)_{[-k,0]}},$
then $$\Hom_{\D\,(\Lambda)}(Z^{\bullet}, -\, [\ell])|_{\D\,(\Lambda)_{[-k,0]}} \simeq \Hom_{\D\,(\Lambda)}(W^{\bullet}, - \,  [\ell])|_{\D\,(\Lambda)_{[-k,0]}}.$$
\end{lem}

\begin{proof} For simplicity, we write $(-,-)$ instead of $\Hom_{\D\,(\Lambda)}(-, -)|_{\D\,(\Lambda)_{[-k,0]}}.$ Suppose that $(Z^{\bullet}, -\, [\ell])\not\simeq (W^{\bullet}, - \,  [\ell]).$ From the following distinguished triangles
$$Z^{0}[0] \rightarrow  Z^{\bullet} \rightarrow  \tau_{\leq -1}(Z^{\bullet})\to Z^0[1]   \, \, \, \,
\text{and} \, \,\, \, \, \,
W^{0}[0] \rightarrow  W^{\bullet} \rightarrow  \tau_{\leq -1}(W^{\bullet})\to W^0[1],$$
we get the following exact sequences of functors 
$$\begin{matrix}
(Z^{0}[1], - \, [\ell]) \rightarrow (\tau_{\leq -1}(Z^{\bullet}), - \, [\ell]) \rightarrow
 (Z^{\bullet}, - \, [\ell])  \rightarrow (Z^{0}[0], - \, [\ell]),\\{}\\
(W^{0}[1], - \, [\ell]) \rightarrow (\tau_{\leq -1}(W^{\bullet}), - \, [\ell]) \rightarrow
 (W^{\bullet}, - \, [\ell]) \rightarrow (W^{0}[0], - \, [\ell]).
\end{matrix}$$
\noindent From the Lemma \ref{projetivo}, we have that $(Z^{0}[1],
- \,  [\ell])= (Z^{0}[0], - \, [\ell])= (W^{0}[1], - \,
[\ell])= (W^{0}[0], - \, [\ell])= 0.$ Therefore $(\tau_{\leq
-1}(Z^{\bullet}), - \, [\ell]) \simeq (Z^{\bullet}, - \,
[\ell])$ and $(\tau_{\leq -1}(W^{\bullet}), - \, [\ell]) \simeq
(W^{\bullet}, - \, [\ell])$. By our assumption, we get that  
 $$(\tau_{\leq -1}(Z^{\bullet}), - \, [\ell])\not\simeq
(\tau_{\leq -1}(W^{\bullet}), - \,[\ell]).$$
We proceed by induction as follows. The step $i=1$ is given by the above statement.  
So, by the inductive step, we can assume that $(\tau_{\leq
-i}(Z^{\bullet}), - \, [\ell])\not\simeq (\tau_{\leq -i}(W^{\bullet}), - \, [\ell])$ for any 
$1<i\leq k-1.$ By applying the functor $( \, \, , - \,[\ell])$ to the distinguished triangles
$Z^{-i}[i] \rightarrow  \tau_{\leq -i}(Z^{\bullet}) \rightarrow
\tau_{\leq -i-1}\,(Z^{\bullet})\to Z^{-i}[i+1]$ and $W^{-i}[i] \to \tau_{\leq -i}(W^{\bullet})
\rightarrow\tau_{\leq -i}(W^{\bullet})\to W^{-i}[i+1],$ and using that 
$(Z^{-i}[i+1], - \, [\ell])=
(W^{-i}[i+1], - \, [\ell])= (Z^{-i}[i], - \, [\ell])=
(W^{-i}[i], - \, [\ell])= 0,$ we get that
$$(\tau_{\leq -i-1}\,(Z^{\bullet}), - \, [\ell])\not\simeq
(\tau_{\leq -i-1}\,(W^{\bullet}), - \, [\ell]).$$ Thus, we obtain
$(Z^{-k}[k], - \,[\ell])\not\simeq (W^{-k}[k], -
\,[\ell]);$ which is a contradiction. So, our assumption is not true, and hence the result 
follows.
\end{proof}

\begin{lem} \label{lema4} Let $\ell>0$ and $k\geq 0;$ and let $K$, $S \in \modu\,(\Lambda)$. If $\Ext^{\ell}_\Lambda(K, -\, )\simeq 
\Ext^{\ell}_\Lambda(S, -\,)$ then $\Hom_{\D\,(\Lambda)}(K,- \, [\ell])|_{\D\,(\Lambda)_{[-k,0]}} \simeq \Hom_{\D\,(\Lambda)}(S, - \, [\ell])
|_{\D\,(\Lambda)_{[-k,0]}} .$
\end{lem}

\begin{proof} By \cite{BM} (see Section 1 and Section 2), we have that there is an isomorphism $\mu_{Y^{\bullet}}:Y^{\bullet}\to 
L(Y^{\bullet})$ in $\D\,(\Lambda)_{[-k,0]},$ which is natural on the variable $Y^{\bullet};$ and moreover 
$L(Y^{\bullet})^{i}\in\inj\,(\Lambda)\;$ $\forall\,i\in[-k,-1].$
\

Let $Y^{\bullet}\in \D\,(\Lambda)_{[-k,0]}.$ Then, by the discussion, given above, it can be assumed that  
$Y^{i}\in\inj\,(\Lambda)\;$ $\forall\,i\in[-k,-1].$ By applying the functors 
$$(K, - ):=\Hom_{\D\,(\Lambda)}(K,-)|_{\D\,(\Lambda)_{[-k,0]}}\text{ and } (S, - ):=\Hom_{\D\,(\Lambda)}(S,-)|_{\D\,(\Lambda)_{[-k,0]}}$$ to 
the distinguished triangle $Y^{0} [\ell] \rightarrow Y^{\bullet}
[\ell]\rightarrow \tau_{\leq -1}(Y^{\bullet})[\ell]\rightarrow Y^{0} [\ell],$ we obtain the following exact sequences
$$\begin{matrix} (K, \tau_{\leq -1}(Y^{\bullet})[\ell-1])
\rightarrow  (K, Y^{0}[\ell]) \rightarrow  (K, Y^{\bullet}[\ell])
\rightarrow  (K, \tau_{\leq -1}(Y^{\bullet})[\ell]),\\{}\\
(S, \tau_{\leq -1}(Y^{\bullet})[\ell-1])
\rightarrow  (S, Y^{0}[\ell]) \rightarrow  (S, Y^{\bullet}[\ell])
\rightarrow  (S, \tau_{\leq -1}(Y^{\bullet})[\ell]).
\end{matrix} $$
By Lemma \ref{projetivo}, we have that the following equalities hold $(K, \tau_{\leq -1}(Y^{\bullet})[\ell-1])= (S,
\tau_{\leq -1}(Y^{\bullet})[\ell-1])= (K, \tau_{\leq
-1}(Y^{\bullet})[\ell])= (S, \tau_{\leq -1}(Y^{\bullet})[\ell])=0;$ and since\\ $\Ext^{\ell}_\Lambda(K, -\, )\simeq 
\Ext^{\ell}_\Lambda(S, -\,),$ it follows that $(K, - \, [\ell]) \simeq (S,- \, [\ell]),$ proving the result.
\end{proof}

\begin{cor}\label{coro} Let $k$ and $\ell$ be integers such that $0 < k< \ell;$ and let $Z^{\bullet}, W^{\bullet}\in \D\,(\Lambda)_{[-k,0]}$ 
be such that $Z^{i},W^{i}\in\proj\,(\Lambda)\;$ $\forall\,i\in[-k+1,0].$\\  If
$\Ext^{\ell-k}_{\Lambda}(Z^{-k}, -\, ) \simeq \Ext^{\ell-k}_{\Lambda}(W^{-n}, -\,)$
then $$\Hom_{\D\,(\Lambda)}(Z^{\bullet},-\, [\ell])|_{\D\,(\Lambda)_{[-k,0]}} \simeq \Hom_{\D\,(\Lambda)}(W^{\bullet},- \,[\ell])|_{
\D\,(\Lambda)_{[-k,0]}}.$$
\end{cor}
\begin{proof} The proof follows from Lemma \ref{lema2} and Lemma \ref{lema4}.
\end{proof}
\vspace{.3cm}

We recall now some definitions (see \cite{R}). A tilting complex over an artin algebra $A$ is a
complex $T^{\bullet} \in \K^{b}\,(\proj\,(A))$ such that $\Hom_{\K^{b}\,(\proj\,(A))}(T, T[n])=0$ for
all integers $n \neq 0$, and $\add\,(T^{\bullet})$, the full subcategory of direct summands of
finite direct sums of copies of $T^{\bullet}$, generates $\K^{b}\,(\proj\,(A))$ as a  triangulated
category. Two artin algebras $A$ and $B$ are called derived equivalent if $\D\,(A)$ and $\D\,(B)$ are equivalent as triangulated categories. For further properties of derived equivalent artin algebras, the reader can see in \cite{R}.
\

In all that follows, we shall consider the following situation: Let $A$ be an artin algebra,
$T^{\bullet} \in \K^{b}\,(\proj\,(A))$ a tilting complex, $B:=\End_{\D\,(A)}(T^{\bullet})^{op}$ and
$F: \D\,(B) \rightarrow \D\,(A)$ be an equivalence of triangulated categories such that $F(B)= T^{\bullet}$. The quasi-inverse of $F$ is denoted 
by $G:\D\,(A)\rightarrow \D\,(B).$ It is well known that $F$ induces equivalences $\D^{b}\,(B)\to \D^{b}\,(A)$ and 
$\K^{b}\,(\proj\,(B))\to \K^{b}\,(\proj\,(A));$ and furthermore, without loss of generality, it can be assumed that 
$T^{\bullet}\in \Com\,(\proj\,(A))_{[-n,0]}$ for some $n\geq 0.$
\

The following results are well known in the literature, for a proof the reader can see in
\cite{K,PX}.

\begin{lem}\cite{K,PX} \label{lema1} Let $B:=\End_{\D\,(A)}(T^{\bullet})^{op},$ where
$T^{\bullet}\in \Com\,(\proj\,(A))_{[-n,0]}$ is a tilting complex as above. Then, the following statements hold.
\begin{itemize}
\item[(a)] There is a tilting complex $Q^{\bullet} \in \Com\,(\proj\,(B))_{[0,n]}$ such that $G(A)\simeq Q^{\bullet}.$
\item[(b)] $H^{\bullet}(FM)\in\Com\,(A)_{[-n,0]}$ for any $M\in\modu\,(A).$
\item[(c)] $H^{\bullet}(GN)\in\Com\,(B)_{[0,n]}$ for any $N\in\modu\,(B).$
\end{itemize}
\end{lem}

\begin{lem} \label{morita} Let $B:=\End_{\D\,(A)}(T^{\bullet})^{op},$ where
$T^{\bullet}\in \Com\,(\proj\,(A))_{[-n,0]}$ is a tilting complex as above. Then, for any
$M\in\modu\,(B),$ the complex $FM$ is quasi-isomorphic to complexes $X^{\bullet}_M$ and $Y^{\bullet}_M,$ where
\begin{itemize}
\item[(a)] $X^{\bullet}_M \in \Com\,(A)_{[-n,0]}$ with $X_M ^i\in\proj\,(A)\;$ $\forall\,i\in[-n+1, 0];$ and
\item[(b)] $Y^{\bullet}_M \in \Com\,(A)_{[-n,0]}$ with $Y_M ^i\in\inj\,(A)\;$ $\forall\,i\in[-n, -1].$
\end{itemize}
Moreover, those quasi-isomorphisms, given above, induce natural isomorphisms in the derived category $\D\,(A).$
\end{lem}
\begin{proof} (a) Consider the complex $P^{\bullet}(M)$ induced by the minimal projective resolution
of $M$. Since $F$ induces an equivalence $\K^{-}\,(\proj\,(B))\to \K^{-}\,(\proj\,(A)),$ it follows that $FP^{\bullet}(M) \in\K^{-}\,
(\proj\,(A)).$ By Lemma \ref{lema1} (b), we have that $H^{\bullet}(FP^{\bullet}(M))\simeq H^{\bullet}
(FM)\in\Com\,(A)_{[-n,0]}.$  Thus, the complex $Z^{\bullet}:=FP^{\bullet}(M)$ is
quasi-isomorphic to the following complex $C^{\bullet}\in\Com\,(A)_{[-n,0]},$ where
$$C^{\bullet}\;:\; \cdots\rightarrow \Coker\,(d^{-n-1})
\rightarrow  Z^{-n+1}\rightarrow  \cdots \rightarrow  Z^{-1}\rightarrow  \Ker\,(d^{0}) \rightarrow
\cdots.$$
\noindent We assert that $\Ker\,(d^{0})\in\proj\,(A).$ Indeed, since $H^{i}(Z^{\bullet})=0$ for $i > 0$, the complex
\begin{equation} \label{complex} \cdots\to 0 \rightarrow  \Ker\,(d^{0}) \rightarrow
 Z^{0} \rightarrow  Z^{1}\rightarrow Z^{2} \rightarrow \cdots
\end{equation}
\noindent  is exact. Moreover, the complex $Z^{\bullet}$ is bounded above and each $Z^{i}$ is
 projective, and so the complex (\ref{complex}) splits; proving that $\Ker\,(d^{0})\in\proj\,(A).$
 Therefore, the complex $X^{\bullet}_M:=C^{\bullet}$ satisfies the desired conditions.
\

(b) This can be proven in a similar way as we did in (a).
\end{proof}

\begin{lem} \label{bras} Let $B:=\End_{\D\,(A)}(T^{\bullet})^{op},$ where
$T^{\bullet}\in \Com\,(\proj\,(A))_{[-n,0]}$ is a tilting complex as above. Then, for any
$S\in\modu\,(A),$ the complex $GS$ is quasi-isomorphic to complexes $X^{\bullet}_S$ and $Y^{\bullet}_S,$ where
\begin{itemize}
\item[(a)] $X^{\bullet}_S \in \Com\,(B)_{[0,n]}$ with $X_S^i\in\proj\,(B)\;$ $\forall\,i\in[1,n];$ and
\item[(b)] $Y^{\bullet}_S \in \Com\,(B)_{[0,n]}$  with $Y_S^i\in\inj\,(B)\;$ $\forall\,i\in[0,n-1].$
\end{itemize}
Moreover, those quasi-isomorphisms induce natural isomorphisms in the derived category $\D\,(B).$
\end{lem}
\begin{proof} It can be done, in an similar way, as we did in \ref{morita}.
\end{proof}

\begin{lem} \label{parte1} Let $B:=\End_{\D\,(A)}(T^{\bullet})^{op},$ where
$T^{\bullet}\in \Com\,(\proj\,(A))_{[-n,0]}$ is a tilting complex as above; and let $M, N \in\modu\,(B).$ If $\Ext_{B}^t (M,\, -\,)
\simeq \Ext_{B}^t(N, \, -\,)$ for all $t \geq \ell > 0$, then
$$\Hom_{\D\,(B)}(M,\, G(-\,[t]))|_{\modu\,(A)} \simeq \Hom_{\D\,(B)}(N, \, G(-\,[t]))|_{\modu\,(A)}\quad\forall\,t \geq n+\ell.$$
\end{lem}
\begin{proof}Let $S\in \modu\,(A).$ Thus, by Lemma \ref{bras} (b), we have the natural isomorphism $GS\stackrel{\sim}{\to} Y^{\bullet}_S$ in 
$\D\,(B),$ where $Y^{\bullet}_S \in \Com\,(B)_{[0,n]}$ and $Y^i_S\in\inj\,(B)\;$ $\forall\,i\in[0,n-1].$ Applying the functors $(M, - )$ and 
$(N, - )$ to the distinguished triangle
$$Y^{n}_S[-n+t] \rightarrow Y^{\bullet}_S [t]\rightarrow \tau_{\leq n-1}(Y^{\bullet}_S) [t] \rightarrow Y^{n}_S[-n+t+1],$$
\noindent we obtain, for any  $t \geq n+\ell,$ the following exact sequences
$$(M, \tau_{\leq n-1}(Y^{\bullet}_S)[t-1]) \rightarrow(M, Y^{n}_S[t-n])
\rightarrow (M, Y^{\bullet}_S[t]) \rightarrow (M, \tau_{\leq n-1}(Y^{\bullet}_S)[t]),$$
$$(N, \tau_{\leq n-1}(Y^{\bullet}_S)[t-1]) \rightarrow(N, Y^{n}_S[t-n])
\rightarrow (N, Y^{\bullet}_S[t]) \rightarrow (N, \tau_{\leq n-1}(Y^{\bullet}_S)[t]).$$
\noindent By Lemma \ref{projetivo}, we have that the following equalities hold 
 $(M, \tau_{\leq n-1}(Y^{\bullet}_S)[t-1])=(M, \tau_{\leq n-1}(Y^{\bullet}_S)[t])=
(N, \tau_{\leq n-1}(Y^{\bullet}_S)[t-1])=(N, \tau_{\leq n-1}(Y^{\bullet}_S)[t])= 0,$ and thus we get that
$(M, Y^{n}_S[t-n]) \simeq (M, Y^{\bullet}_S[t])$ and $(N, Y^{n}_S[t-n]) \simeq (N, Y^{\bullet}_S[t])$ for $t \geq n+\ell.$
Therefore $(M,\, G(-\,[t]))|_{\modu\,(A)} \simeq (N, \, G(-\,[t]))|_{\modu\,(A)}$ for all  $t \geq n+\ell$.
\end{proof}

For an artin algebra $\Lambda,$ we denote the $\phi$ dimension of a $\Lambda$-module $M$ by $\phi_{\Lambda}(M).$

\begin{teo}\label{tititi} Let $A$ and $B$ be artin algebras, which  are derived equivalent. Then, $\Fidim\,(A)< \infty$ if and only if 
$\Fidim\,(B) < \infty$.
More precisely, if $T^{\bullet}$ is a tilting complex over $A$ with $n$ non-zero terms and such that $B\simeq\End_{\D\,(A)}(T^{\bullet})$, then
$$\Fidim\,(A)-n\leq \Fidim\,(B)\leq \Fidim\,(A)+n.$$
\end{teo}

\begin{proof} Assume that $\Fidim\,(A)= r < \infty $. If $\Fidim\,(B)\leq n$ then $\Fidim\,(B) \leq \Fidim\,(A)+ n $. Suppose now that there 
exists $J \in \modu\,(B)$ such that $\phi_{B}(J)=d>n$. It follows from Theorem \ref{FidimExt} that $J= \hat{M} \oplus \hat{N}$ and there exist 
$M \in \add\,(\hat{M})$ and $N\in\add\,(\hat{N})$ such that
\begin{equation}\label{eq:ext}
\left\{\begin{array}{cccc}
\Ext^{d}_{B}(M,-) & \not\simeq & \Ext^{d}_{B}(N, -),& \\
\Ext^{t}_{B}(M, -) & \simeq & \Ext^{t}_{B}(N, -)& \text{ for }\;t \geq d+1.
\end{array} \right. \end{equation}
\noindent Using Lemma \ref{parte1}, for the second equation of (\ref{eq:ext}), we have

\begin{equation}\label{eq:equG}
\left\{ \begin{array}{cccc}
\Hom_{\D\,(B)}(M, -[d])|_{\modu\,(B)} & \not\simeq & \Hom_{\D\,(B)}(N, -[d])|_{\modu\,(B)},& \\
\Hom_{\D\,(B)}(M, G( -[t]))|_{\modu\,(A)} & \simeq & \Hom_{\D\,(B)}(N, G( -[t]))|_{\modu\,(A)}&
\end{array} \right. \end{equation}
\noindent for all $t \geq n+d+1$.

\noindent Applying the equivalence $F$ in (\ref{eq:equG}), we get
\begin{equation}\label{eq:eq2}
\left \{\begin{array}{ccc}
\Hom_{\D\,(A)}(FM, F(-[d]))|_{\modu\,(A)} & \not\simeq & \Hom_{\D\,(A)}(FN, F(-[d]))|_{\modu\,(A)}, \\
\Hom_{\D\,(A)}(FM, -[t])|_{\modu\,(A)} & \simeq & \Hom_{\D\,(A)}(FN, -[t])|_{\modu\,(A)}
\end{array} \right. \end{equation}
\noindent for all $t \geq n+d+1$.

By Lemma \ref{morita}, the complexes $FM$ and $FN$ can be replaced, respectively, by $Z^{\bullet}, W^{\bullet}\in\Com\,(A)_{[-n,0]}$ such 
that  $Z^{i},W^{i}\in\proj\,(A)\;$ $\forall\,i\in [-n+1, 0]$. It follows, from Lemma  \ref{lema3}, that the second
item of (\ref{eq:eq2}) is equivalent to
\begin{equation}\label{eq:equacao3}
\Ext_{A}^t(Z^{-n}, -) \simeq  \Ext_{A}^t(W^{-n}, -)
\end{equation}
\noindent  for all $t \geq n+d+1$.

By Corollary \ref{coro}, the first  item of (\ref{eq:eq2}) give us that
\begin{equation}\label{eq:equac4}
\Ext^{d-n}_{A}(Z^{-n}, -)  \not\simeq  \Ext^{d-n}_{A}(W^{-n}, -).
\end{equation}

%
In particular, we get from (\ref{eq:equac4}) that $Z^{-n}\not\simeq W^{-n}.$ We can assume that 
$\add\,(Z^{-n})$ and $\add\,(W^{-n})$ have trivial intersection because, otherwise, we can decompose $Z^{-n}$ and $W^{-n}$ as direct sum of 
indecomposables and withdraw from each one the common factors. The modules obtained satisfy (\ref{eq:equacao3}) and  (\ref{eq:equac4}), and 
their additive closure has trivial intersection.
\

Let $L:= Z^{-n} \oplus W^{-n}$. It follows from Theorem \ref{FidimExt}, (\ref{eq:equacao3}) and  (\ref{eq:equac4}) that 
$ d-n \leq \phi_{A} (L) \leq d+n.$ Since $\Fidim\,(A)=r.$ we have $d \leq n+r$ and hence $\Fidim\,(B)\leq n+ \Fidim\,(A)$. Analogously, 
it can be shown that $\Fidim\,(A)\leq  \Fidim\,(B) +n.$
\end{proof}

As an application of the main result, we get the following generalisation of the classic Bongartz's 
result \cite[Corollary 1]{B}. In order to do that, we recall firstly, the notion of tilting module.  
Following Y. Miyashita in \cite{M}, it is said that an $A$-module  $T\in\modu\,(A)$ is a tilting module, if $T$ satisfies the following properties: (a) $\pd\,T$ is finite, (b) $\Ext_A^i(T,T)=0$ for any $i>0$ and (c) there is an exact sequence $0\to {}_AA\to T_0\to T_1\to\cdots\to T_m\to 0$ 
in $\modu\,(A),$ with $T_i\in\add\,(T)$ for any $0\leq i\leq m.$

\begin{cor}\label{coroPdT} Let $A$ be an artin algebra, and let $T\in\modu\,(A)$ be a tilting $A$-module. Then, for the artin algebra $B:=\End_A(T)^{op},$ we have that  
$$\Fidim\,(A)-\pd\,T\leq \Fidim\,(B)\leq \Fidim\,(A)+\pd\,T.$$
\end{cor}
\begin{proof} Let $m:=\pd\,T.$ Then, the minimal projective resolution of $T$ is as follows $0\to P_m\to\cdots\to P_1\to P_0\to T\to 0.$ So, the 
complex $$P^\bullet(T)\;:\;\cdots\to 0\to P_m\to\cdots\to P_1\to P_0\to 0\to\cdots$$ is a tilting complex in $\D\,(A)$ and also 
$B\simeq \End_{\D\,(A)}(P^\bullet(T))^{op}.$ Thus, by \ref{tititi} we get the result, since in this case $n$ is $\pd\,T.$ 
\end{proof}

\begin{remark} Let $A$ be an artin algebra. In \cite[Corollary 1]{B}, Bongartz consider a classical tilting module, that is, a tilting module $T\in\modu\,(A)$ such that $\pd\,T\leq 1.$ The original Bongartz's result says that $\gldim\,(B)\leq \gldim\,(A)+1,$ where $B:=\End_A(T)^{op}.$
\end{remark}

Finally, as another application of the main theorem, we get the following result for one-point 
extension algebras.

\begin{cor} Let $A$ and $B$ be two finite-dimensional $k$-algebras, $M \in \modu\,(A)$ and $N \in \modu\,(B)$. Let $A[M]$ and 
$B[N]$ be the respective one-point extensions. If $A$ and $B$ are derived-equivalent, then the finiteness of the $\phi$-dimension of one of 
the algebras $A, B, A[M]$ and $B[N]$ implies that all of them have finite $\phi$-dimension.
\end{cor}
\begin{proof} Let $X \in \modu\,(A[M])$. We know by Lemma \ref{itphi} that $\phi(X) \leq 1+ \phi(\Omega X)$. Since $\Omega X$ can be seen as 
an $A$-modulo, it follows that $\Fidim\,(A)< \infty$ implies $\Fidim\,(A[M])< \infty$. Now, the fact that $\modu\,(A)\subseteq \modu\,(A[M])$ 
gives us the other implication. From this fact and Theorem \ref{tititi}, we get that the following equivalences hold: 
$\Fidim\,(A[M])< \infty \Leftrightarrow \Fidim\,(A)< \infty \Leftrightarrow  \Fidim\,(B)< \infty \Leftrightarrow \Fidim\,(B[M])< \infty$.
\end{proof}


\vskip3mm \noindent S\^{o}nia Maria Fernandes:\\ Universidade Federal de Vi\c{c}osa,\\
Departamento de Matem\'atica. CP 36570000, Vi\c{c}osa, MG, Brasil.\\
{\tt somari@ufv.br}

\vskip3mm \noindent Marcelo Lanzilotta:\\ Instituto de Matem\'atica y Estad\'{i}stica Rafael Laguardia,\\
J. Herrera y Reissig 565, Facultad de Ingenier\'{i}a, Universidad de la Rep\'ublica. CP 11300, Montevideo, URUGUAY.\\
{\tt marclan@fing.edu.uy}

\vskip3mm \noindent Octavio Mendoza Hern\'andez:\\ Instituto de Matem\'aticas, Universidad Nacional Aut\'onoma de M\'exico.\\
Circuito Exterior, Ciudad Universitaria, C.P. 04510, M\'exico, D.F. M\'EXICO.\\
{\tt omendoza@matem.unam.mx}

\end{document}